\documentclass{amsart}
\usepackage{amsmath}
\usepackage{amsfonts}

\newtheorem{theorem}{Theorem}
\numberwithin{theorem}{subsection} 
\newtheorem*{thma}{Theorem A}
\newtheorem*{thmb}{Theorem B}
\newtheorem*{thmc}{Theorem C}
\newtheorem*{thmd}{Theorem D}
\newtheorem*{thme}{Theorem E}
\theoremstyle{plain}

\newtheorem{corollary}{Corollary}

\newtheorem{definition}{Definition}

\newtheorem{remark}{Remark}

\numberwithin{equation}{section}

\begin{document}
\title[ Bohr's inequality for stable mappings ]{On the Bohr's inequality for stable mapping}
\author{ Zayid AbdulHadi}
\address{Department of Mathematics and statistics\\
American University of Sharjah\\
Sharjah, Box 26666\\
UAE}
\email{zahadi@aus.edu}
\author{ Layan El Hajj}
\address{Department of Mathematics, American University of  Dubai, Dubai,\\
UAE. }
\email{ lhajj@aud.edu }
\date{June 3, 2021}

\begin{abstract}

We consider the class of \emph{stable} harmonic mappings $f=h+\overline{g}$  introduced by Martin, Hernandez, and the class of \emph{stable} logharmonic mappings $f=zh\overline{g}$ introduced by AbdulHadi, El-Hajj. We determine Bohr's radius for the classes of stable univalent harmonic mappings, stable convex harmonic mappings and stable univalent logharmonic mappings. We also consider improved and refined versions of Bohr's inequality and discuss the Bohr's Rogonsiski radius for these family of mappings.
\end{abstract}

\maketitle

\section{Introduction}

A planar harmonic mapping in the unit disk $U$ is a complex-valued harmonic function $f$ which maps $U$ onto some planar domain $f(U)$ and satisfies $\Delta f = \partial_{z\overline{z}} f= 0 $. Since $U$ is simply connected, the mapping $f$ has a canonical decomposition $f=h+\overline{g}$, where $h$ and $g$ are analytic in $U$.  

A logharmonic mapping defined on the unit disk $U=\{z:|z|<1\}$ is a solution
of the nonlinear elliptic partial differential equation 
\begin{equation}
\frac{\overline{f_{\overline{z}}}}{\overline{f}}=a\frac{f_{z}}{f},
\label{eq1.1}
\end{equation}%
where where $a$ is an analytic function satisfying $|a(z)|<1$ in $U.$


If $f$ is a non-constant logharmonic mapping of $U$ and vanishes only at $%
z=0 $, then $f$ admits the representation 
\begin{equation}
f(z)=z^{m}|z|^{2\beta m}h(z)\overline{g(z)},  \label{eq1.2}
\end{equation}%
where $m$ is a non-negative integer, $\mathrm{Re}(\beta )>-1/2$, and $h$ and $%
g$ are analytic functions in $U$ satisfying $g(0)=1$ and $h(0)\neq 0$ (see
\cite{AB2}). 
Note that $f(0)\neq 0$ if and only if $m=0$, and that a univalent
logharmonic mapping on $U$ vanishes at the origin if and only if $m=1$, that
is, $f$ has the form 
\begin{equation*}
f(z)=z|z|^{2\beta }h(z)\overline{g(z)},
\end{equation*}%
where $\mathrm{Re}(\beta )>-1/2$ and $0\notin (hg)(U).$ This class has been
studied extensively (for details see \cite{AA1}-\cite{AH2}). In \cite{HM}, the authors introduced the  families of stable harmonic mappings. 
\begin{definition}
A (sense preserving) harmonic mapping $f=h+ \overline{g}$ is stable harmonic univalent or SHU in the unit disk (resp.stable harmonic convex (SHC) if all the mappings $f_{\lambda} =h+ \lambda g$ with $|\lambda|= 1$ are univalent (resp.convex) in $U$.
\end{definition}

Similarly, in \cite{AE}, the authors introduced the notion of stable logharmonic mappings and proved some interesting properties for this class
of functions. We first recall the definition of stable univalent logharmonic mappings, and the properties that will be useful for our purposes.

\begin{definition}
A logharmonic mapping $f=zh\overline{g}$  such that $f (0) = 0$ and $h(0) = g(0) = 1,$ is said to be stable univalent
logharmonic or SST$_{Lh}$ in the unit disk, if the mappings $\ f_{\lambda
}=zh(z)\overline{g(z)}^{\lambda }$ are univalent for all $|\lambda |\leq 1.$
\end{definition}

In this paper we will consider Bohr's Radius for both of these classes of functions. Bohr's inequality says that if $f(z)=\underset{n=0}{\overset{\infty}{{\displaystyle\sum}}}a_{n}z^{n}\ $ is analytic in the unit disc $U$ and $|f(z)|<1$ for all $z$ in $U,$ then $\underset{n=0}{\overset{\infty}{{\displaystyle\sum}}}\left\vert a_{n}z^{n}\right\vert \leq1\ $for all $z\in$ $U$ with $|z|\leq\frac{1}{3}.$ This inequality was discovered by Bohr in 1914 (see \cite{Bo}).
Bohr actually obtained the inequality for $|z|\leq\frac{1}{6}.$ Later Wiener, Riesz and Schur independently, established the inequality for $|z|\leq \frac{1}{3}$ and showed that $\frac{1}{3}$ is sharp (see \cite{PPS, S2, T}).  An observation shows that the quantity $1-|a_0|$ is equal to $d(f(0), \partial f(U))$. Therefore, Bohr's inequality, can be written in the following form
$$\underset{n=1}{\overset{\infty}{{\displaystyle\sum}}}\left\vert a_{n}z^{n}\right\vert \leq 1- |a_0| = d(f(0), \partial f(U)),$$
for $|z| = r \leq 1/3$, where $d$ is the Euclidean distance. It is important to note that the
constant $1/3$ is independent of the coefficients of the Taylor series of $f(z)$. In recent years, a number of researchers revisited the work of Bohr improving and extending this work to more general settings.
Various generalizations of the classical Bohr inequality have been investigated in different branches of mathematics,
 See \cite{AA,AAH,BDK,IKKP}, and  the paper \cite{A3} carried Bohr's theorem to prominence for the case of several complex variables.
We refer the reader to a survey on this topic by Abu-Muhanna et al. \cite{AAP} and the references therein, which increased the interest in the topic and was followed by a series of papers by several authors generalizing this and several improved and refined versions of Bohr's inequality have been discussed, see for instance \cite{AB, DR,KP1,KP2,KP3,KP4,KPS,LP2,EPR}.


Besides the Bohr radius, there is also the notion of Rogosinski radius (see \cite{LG,SS}) which is described as follows: If $f(z)=\underset{n=0}{\overset{\infty}{{\displaystyle\sum}}}a_{n}z^{n}\ $ is an analytic function on $U$ such that $|f(z)| < 1$ in U, then for every $N \geq 1$, we have $|S_N (z)| < 1$ in the disk $ |z| < 1/2$ and this radius is sharp, where $S_N (z) =\underset{n=0}{\overset{N}{{\displaystyle\sum}}}a_{n}z^{n}$ denotes the partial sums of $f$. In \cite{KKP,KP4}, the authors define the Bohr-Rogosinski sum $R_{N}^{f}(z)$ by
$$R_{N}^{f} (z) = |f(z)| + \underset{n=N}{\overset{\infty}{{\displaystyle\sum}}}\left\vert a_{n}z^{n}\right\vert,\,\,\,\,|z|=r.$$
It is worth noting that for $N = 1$, this quantity is related to the classical Bohr sum in which
$f(0)$ is replaced by $f(z)$. The Bohr-Rogosinski's radius is defined to be the largest number $r_0>0$ such that $R_N^{f} (z) \leq 1, $ for $|z|\leq r_0.$ Moreover we have $|S_N(z)|\leq R_N^f(z)$, hence Bohr-Rogosinski's sum is related to Rogosinski's
characteristic and the validity of Bohr-type radius for $R_N^f(z)$ gives Rogosinski radius in the case
of bounded analytic functions. Another version of Bohr's theorem exists with the initial coefficient $|a_0|^2$  instead of $|a_0|$, and it is well-known that the Bohr radius with this change of initial coefficient becomes $1/2$ instead of
$1/3$. We will also look at the analogue in the setting of Bohr-Rogosinski, i.e., with $|f(z)|^2$ in place of $|f(z)|$.

This paper is organized as follows: in section 2 we consider the Bohr type inequalities for the  the class of stable harmonic mappings and in section 3 we will consider the class of stable logharmonic mappings. In each of these sections, we will show for our classes of stable mapping a version of the Bohr inequality along with its improved and refined versions as in \cite{EPR,KP4,PVW}, in addition, the Bohr Rogosinski radius as in \cite{KKP,KP2} will be discussed.

\section{Bohr inequalities for stable harmonic mappings}

There has been a lot of interest recently in Bohr's radius for harmonic mappings. We refer to the survey paper \cite{KKP} on a nice compilation and exposition of these results. Namely, results have been obtained for locally univalent harmonic mappings and $k$-quasiconformal harmonic mappings.

The class $S_{H}^{0}$ is defined  as the family of sense preserving univalent harmonic mappings $f=h+\overline{g}$ in the unit disk with the normalizations $h(0) =g(0) = 1-h^{'}(0) =g^{'}(0) = 0.$  
We  let  
\begin{equation*}
h(z)= z + \sum\limits_{n=2}^{\infty }a_{n}z^{n}, and \,\,\,
g(z)= \sum\limits_{n=2}^{\infty }b_{n}z^{n} .
\end{equation*} 

We define the majorant series for  $f=h+\overline{g}$ as in \cite{AAH}  to be 
\begin{equation*}
M_{f}(r)=  \sum\limits_{n=1}^{\infty}\left(|a_{n}+|b_{n}|\right) r^{n}= r + \sum\limits_{n=2}^{\infty}\left(|a_{n}+|b_{n}|\right) r^{n}
\end{equation*}

\subsection{Bohr's radius for stable harmonic mappings}

We consider the Bohr's radius for the class of stable univalent harmonic mappings and stable convex harmonic mappings in $S_H^0$, which were introduced in \cite{HM}.  We will need the following coefficients and distortion theorems for these class of functions :
 
\begin{thma}({\cite{HM}})\label{coefficient}
\begin{itemize}
 \item[(i)]Assume that  $f=h+\overline{g}\ $ in $S_H^0 $ is stable univalent harmonic mapping. Then for all non-negative integers $n$, we have 
\begin{equation}\label{coeffunivalent}
\left\vert |a_{n}|-|b_{n}|\right\vert \leq \max \{|a_{n}|,|b_{n}|\}\leq|a_{n}|+|b_{n}|\leq n.
\end{equation}

\item[(ii)]Assume that  $f=h+\overline{g}\ $ in $S_H^0 $ is stable convex harmonic mapping. Then for all nonnegative integers $n$, we have 
\begin{equation}\label{coeffconvex}
\left\vert |a_{n}|-|b_{n}|\right\vert \leq \max \{|a_{n}|,|b_{n}|\}\leq|a_{n}|+|b_{n}|\leq 1.
\end{equation}
\end{itemize}
\end{thma}

 \begin{thmb}({\cite{HM}})\label{distortion}
\begin{itemize}
\item[(i)] Let $f=h+\overline{g}$  in $S_H^0 $ be a stable univalent harmonic mapping on the unit disk $U$, then for all $z \in U,$
 we have:%
\begin{equation*}
\dfrac{|z|}{(1+|z|)^{2}}\leq |f(z)|\leq \dfrac{|z|}{(1-|z|)^{2}}.
\end{equation*}

\item[(ii)]Let $f=h+\overline{g}$ in $S_H^0 $ be a stable convex harmonic mapping on the unit disk $U$,
then for all z$\ \in U,$ we have:%
\begin{equation*}
\dfrac{|z|}{1+|z|}\leq |f(z)|\leq \dfrac{|z|}{1-|z|}.
\end{equation*}
\end{itemize}
\end{thmb}



We are now ready to prove the Bohr radius for stable harmonic mappings.

\begin{theorem}\label{HarmonicBohr}

\begin{enumerate}
\item[(i)] Let $f=h+\overline{g}$  in $S_H^0$  be a stable convex harmonic mapping on the unit disk U, then $$M_{f}(r)\leq d(f(0),\partial f(U))$$ if \ $|z|\leq r_{0}= 1/3,$
where $r_{0}$ is the unique root in $(0,1)$ of \ 
\begin{equation*}
 \frac{r}{1-r } =\frac{1}{2}.
\end{equation*}
\item[(ii)]Let $f=h+\overline{g}$ in $S_H^0$ be a stable univalent harmonic mapping (or stable starlike) on the unit disk $U$, then 
$$M_{f}(r)\leq 1 $$ for \ $|z|\leq r_{0},$
where $r_{0}\approx 0.382$ is the unique root in $(0,1)$ of
\begin{equation*}
  \frac{r}{\left( 1-r\right) ^{2}} =1,
\end{equation*}
and 
 $$M_{f}(r)\leq d(f(0),\partial f(U))$$ if $|z|\leq r_{0},$
where $r_{0}\approx 0.172$ is the unique root in $(0,1)$ of \ 
\begin{equation*}
  \frac{r}{\left( 1-r\right) ^{2}} =\frac{1}{4}.
\end{equation*}
In all of the above $r_0$ is the best possible radius.
\end{enumerate}
\end{theorem}

\begin{proof}
For $|z|=r$, in the case where $f$ is stable convex harmonic we have 
\begin{equation}\label{Mrconvex}
M_{f}(r)\leq   \sum\limits_{n=1}^{\infty }r^{n} = \frac{r}{1-r}.
\end{equation}

But, \begin{equation}\label{dist2}
d(f(0),\partial f(U))=\underset{|z|\rightarrow 1}{\lim }\inf \ |f(z)-f(0)|\geq \underset{|z|\rightarrow 1}{\lim }\inf \frac{|z|}{1+|z|} =\frac{1}{2}
\end{equation}%
Hence 

$$M_{f}(r)<\frac{1}{2}\leq d(f(0),\partial f(U))$$ if  $|z|\leq r_{0}=1/3,$ where $r_{0}$ is the unique root in $(0,1)$ of \ 
\begin{equation*}
 \frac{r}{ 1-r} =\frac{1}{2}.
\end{equation*}
A suitable rotation of the analytic mapping $g(z)=\dfrac{z}{(1-z)}$ shows that $r_0$ is the best possible radius, since for this function we have $M_{g}(r) = \frac{r}{\left( 1-r\right) }$ and $d(g(0),\partial g(U))=1/4.$ \\
Similarly for part (ii), we have  for $f$ stable univalent harmonic
\begin{equation}\label{Mrunivalent}
M_{f}(r)\leq  \sum\limits_{n=1}^{\infty }\left(|a_{n}|+|b_{n}|\right) r^{n} \leq 
\sum\limits_{n=1}^{\infty }nr^{n} = \frac{r}{\left(1-r\right) ^{2}}.
\end{equation}

It follows that $M_{f}(r)<1~$for $|z|\ \leq r_{0}=0.382,$ where $r_{0}$ is
the unique root in $(0,1)$ of \ 

\begin{equation*}
\frac{r}{\left(1-r\right) ^{2}} =1.
\end{equation*}
We next note that  
\begin{equation}\label{dist1}
d(f(0),\partial f(U))=\underset{|z|\rightarrow 1}{\lim }\inf \ |f(z)-f(0)|=\underset{|z|\rightarrow 1}{\lim }\inf \ \dfrac{|f(z)-f(0)|}{|z|}\geq \underset{|z|\rightarrow 1}{\lim }\inf \frac{1}{(1+|z|)^{2}} =\frac{1}{4},
\end{equation}%
where the last inequality follows from Theorem B.\\
Hence we have $$M_{f}(r)<\frac{1}{4}\leq d(0,\partial f(U))$$ if \ $|z|\leq r_{0}=0.172,$ where $r_{0}$ is the unique root in $(0,1)$ of 
\begin{equation*}
 \frac{r}{\left( 1-r\right) ^{2}} =\frac{1}{4}.
\end{equation*}
A suitable rotation of the analytic Koebe mapping $k(z)=\dfrac{z}{(1-z)^2}$ shows that $r_0$ is the best possible radius, since for this function we have $M_r(k) = \frac{r}{\left( 1-r\right) ^{2}}$ and $d(k(0),\partial k(U))=1/4.$ \\

\end{proof}

\subsection{Improved Bohr's inequality for stable harmonic mappings}
Kayumov and Ponnusamy \cite{KP4} improved the classical version of the Bohr theorem in
four different formulations. Later in a survey article in \cite{KP4}, the authors have further improved a couple of these results. In the same spirit, as with the analytic case, Evdoridis et al. \cite{EPR}  improved the results of \cite{KP4} for locally univalent harmonic mappings. We refer to \cite{KKP2} for a nice exposition of all these results. 
In this paper, we show a version of the Improved Bohr's inequality under the stability condition for the harmonic mappings by adding a suitable non-negative term at the left hand side of the inequality.

\begin{theorem}\label{improvedBohrHarmonic}
\begin{enumerate}
\item[(i)] Let $f=h+\overline{g}$ in $S_H^0$  be a stable convex harmonic mapping on the unit disk $U$, and Let $S_r$ be the area of the
image $f(D_r)$, with $D_r = |z|=r.$ Then,

$$M_{f}(r)+ \left(\frac{S_r}{\pi}\right)^k\leq d(f(0),\partial f(U))$$ if \ $|z|\leq r_{0},$
where $r_{0}$ is the unique root in $(0,1)$ of \ 
\begin{equation*}
  \frac{r}{ 1-r }+\frac{r^{2k}}{(1-r^2)^{2k}} =\frac{1}{2}.
\end{equation*}

In all of the above $r_0$ is the best possible radius.  Note that, For $k=1$, we have $r_0 \approx 0.268$, for $k=10$, we have $r_0 \approx 0.33$.

\item[(ii)]Let $f=h+\overline{g}$ in $S_H^0$  be a stable univalent harmonic mapping (or stable starlike) on the unit disk U, then 
$$M_{f}(r)+  \left(\frac{S_r}{\pi}\right)^k\leq d(f(0),\partial f(U)) $$ for \ $|z|\leq r_{0},$
where $r_{0}\approx 0.382$ is the unique root in $(0,1)$ of
\begin{equation*}
  \frac{r}{\left( 1-r\right)^2 }+\frac{(r^6+4r^4+r^2)^{k}}{(r^2-1)^{4k}} =\frac{1}{4}.
\end{equation*}
Note that, For $k=1$, we have $r_0 \approx 0.157$, for $k=10$, we have $r_0 \approx 0.172$.
\end{enumerate}
\end{theorem}
\begin{proof}
i) We first find need the bound for $S_r$ under the stability condition.

Here we have, 

\begin{align*}
\frac{S_r}{\pi} &=\frac{1}{\pi}\int\int_U J_f dA \\
&=\frac{1}{\pi}\int\int_U \left(|h'|^2 -|g'|^2 \right) dxdy\\
&=\underset{n=1}{\overset{\infty}{{\displaystyle\sum}}}\left( n|a_{n}|^2-n|b_n|^2\right)r^{2n}\\
&=\underset{n=1}{\overset{\infty}{{\displaystyle\sum}}} n (|a_{n}|-|b_n|)(|a_{n}|+|b_n|)r^{2n}.\\
\end{align*}
If $f$ is stable convex we use equation \eqref{coeffconvex} to get that $\frac{S_r}{\pi} \leq \underset{n=1}{\overset{\infty}{{\displaystyle\sum}}} n r^{2n} =\large{ \frac{r^2}{(1-r^2)^2}},$ and it follows that 
$$M_{f}(r)+ \left(\frac{S_r}{\pi}\right)^k \leq \frac{r}{ 1-r }+\frac{r^{2k}}{(1-r^2)^{2k}} \leq d(f(0),\partial f(U))$$ if \ $|z|\leq r_{0},$
where $r_{0}$ is the unique root in $(0,1)$ of \ 
\begin{equation*}
  \frac{r}{ 1-r }+\frac{r^{2k}}{(1-r^2)^{2k}} =\frac{1}{2}.
\end{equation*}
ii) In the case where $f$ is stable univalent or stable starlike we use equation \eqref{coeffunivalent} to get $\frac{S_r}{\pi} \leq \underset{n=1}{\overset{\infty}{{\displaystyle\sum}}} n^3 r^{2n} = \frac{r^6+4r^4+r^2}{(r^2-1)^4},$ and the rest of the proof will follow as in part (i).

\end{proof}

Another version of the improved Bohr inequality is stated in the following theorem. Its proof is ommitted as it follows along the same lines as Theorem \ref{improvedBohrHarmonic}.
\begin{theorem}

\begin{enumerate}
\item[(i)] Let $f=h+\overline{g}$ in $S_H^0$  be a stable convex harmonic mapping on the unit disk $U$, then $$ M_{f}(r)+c \underset{n=1}{\overset{\infty}{{\displaystyle\sum}}}\left( |a_{n}|^k+|b_n|^k\right)r^n  \leq d(f(0),\partial f(U))$$ if \ $|z|\leq r_{0},$ where $k$ and $c$ are constants, and $r_{0} = \frac{1}{3+4c}$ is the unique root in $(0,1)$ of \ 
\begin{equation*}
  (1+2c)\frac{r}{ 1-r } =\frac{1}{2}.
\end{equation*}
Note that, For $c=1$, we have $r_0 \approx 0.142$, for $c=0.1$, we have $r_0 \approx 0.29$.

\item[(ii)]Let $f=h+\overline{g}$ in $S_H^0$  be a stable univalent harmonic mapping (or stable starlike) on the unit disk U, then 
$$M_{f}(r)+ c \underset{n=1}{\overset{\infty}{{\displaystyle\sum}}}\left( |a_{n}|^2+|b_n|^2\right)r^n \leq d(f(0),\partial f(U)) $$ for \ $|z|\leq r_{0},$
where $r_{0}$ is the unique root in $(0,1)$ of
\begin{equation*}
  \frac{r}{\left( 1-r\right)^2 }+2c\frac{(r^6+4r^4+r^2)}{(r^2-1)^{4}} =\frac{1}{4}.
\end{equation*}
Note that, for $c= 0.1 $, we have $r_0 \approx 0.168$ for $c=1$, we have $r_0 \approx 0.148$.
\end{enumerate}
\end{theorem}

\subsection{Refined Bohr's inequality for stable harmonic mappings}
In \cite{PVW}, the authors compared $\underset{n=1}{\overset{\infty}{{\displaystyle\sum}}}|a_{n}|r^n $ with another functional often considered in function theory, namely $\underset{n=1}{\overset{\infty}{{\displaystyle\sum}}}|a_{n}|^2r^{2n} $  that they abbreviated as 
$\left\|f\right\|_r^2$ for the analytic function $f(z) = \underset{n=1}{\overset{\infty}{{\displaystyle\sum}}}a_{n}z^n $. They thus obtain the following refinement of Bohr's inequality:
\begin{thmc}
 Suppose that $f(z) = \underset{n=1}{\overset{\infty}{{\displaystyle\sum}}}a_{n}z^n $ analytic in the unit disk with $a_0=0$ and $|f(z)|\leq 1$ Then
 $$\underset{n=1}{\overset{\infty}{{\displaystyle\sum}}}|a_{n}|r^n +\frac{1}{1-r}\left\|f\right\|_r^2 \leq 1 $$ 
for $r \leq 1/2.$
\end{thmc}

We consider a similar refinement of Bohr's inequality for the family of stable mappings $f=h+\overline{g}$ in $S_H^0$. We leave the proof to the reader as it uses similar ideas as the previous theorems. 

\begin{theorem}
\begin{enumerate}
\item[(i)] Let $f=h+\overline{g}$  in $S_H^0$  be a stable convex harmonic mapping on the unit disk U, then for  $N \in \mathbb{N}$ we have $$M_{f}(r)+\frac{r^N}{1-r^N} \underset{n=1}{\overset{\infty}{{\displaystyle\sum}}}\left(|a_{n}|^2 +|b_{n}|^2\right) r^{2n}\leq d(f(0),\partial f(U))$$ if \ $|z|\leq r_{0},$
where $r_{0}$ is the unique root in $(0,1)$ of the function \ 
\begin{equation*}
 \psi_N(r) = \frac{r}{1-r }+\frac{r^{N+2}}{(1-r^2)(1-r^N)} -\frac{1}{2}.
\end{equation*}
For $N=1,$  $r_{0}\approx 0.29$ and for a larger $N $, $r_{0}\approx 0.33$.

\item[(ii)] Let $f=h+\overline{g}$  in $S_H^0$  be a stable univalent (or stable starlike) harmonic mapping on the unit disk U, then for  $N \in \mathbb{N}$ we have $$M_{f}(r)+\frac{r^N}{1-r^N}\underset{n=1}{\overset{\infty}{{\displaystyle\sum}}}\left(|a_{n}|^2 +|b_{n}|^2\right) r^{2n}\leq d(f(0),\partial f(U))$$ if \ $|z|\leq r_{0},$
where $r_{0}$ is the unique root in $(0,1)$ of the function \ 
\begin{equation*}
 \phi_N(r) = \frac{r}{(1-r)^2 }-\frac{r^{N+2}(1+r^2)}{(r^2-1)^3(1-r^N)} -\frac{1}{4}.
\end{equation*}
For $N>1,$  $r_{0}\approx 0.172$.

In all of the above, $r_0$ is the best possible radius.

\end{enumerate}
\end{theorem}

\subsection{Bohr Rogosinski radius for stable harmonic mappings}
We connect the Bohr Radius and the Bohr Rogosinski radius for stable mappings in the following theorem:
\begin{theorem}\label{rog1}
\begin{enumerate}

\item[(i)] Let $f=h+\overline{g}$ in $S_H^0$  be a stable convex harmonic mapping on the unit disk $U$, then for each $m, N \in \mathbb{N},$ we have
 $$|f(z^m)| +  \underset{n=N}{\overset{\infty}{{\displaystyle\sum}}}\left( |a_{n}|+|b_n|\right)r^n \leq d(f(0),\partial f(U))$$ if $|z|\leq r_{m,N},$
where $r_{m,n}$ is the unique root in $(0,1)$ of  the the function
\begin{equation}\label{functionconvex1}
  H_{m,N} (r) = \frac{2r^N(1-r^m)}{ 1-r}+3r^m-1,
\end{equation}

and $r_{m,N}$ cannot be improved. Here $\underset{N\rightarrow\infty}{\lim}r_{m,N}=\left(\frac{1}{2}\right)^{1/m} ,\,\,$  $\underset{m\rightarrow\infty}{\lim}r_{m,N}=r_N ,$  where $r_N$ is the solution of the equation $r^N=1-r, $ and $\underset{m,N\rightarrow\infty}{\lim}R_{m,N}= 1$.

Moreover, 
$$|f(z^m)|^2 +  \underset{n=N}{\overset{\infty}{{\displaystyle\sum}}}\left( |a_{n}|+|b_n|\right)r^n \leq d(f(0),\partial f(U))$$ if $|z|\leq R_{m,N},$ where $R_{m,n}$ is the unique root in $(0,1)$ of  the function
\begin{equation}\label{functionconvex2}
  K_{m,N} (r) = \frac{2r^N(1-r^m)^2}{ 1-r}+r^{2m}+ 2r^m-1,
\end{equation}
and $R_{m,N}$ cannot be improved. Here $\underset{N\rightarrow\infty}{\lim}r_{m,N}=\left( -1+\sqrt{2}\right)^{1/m} ,\,\,$  $\underset{m\rightarrow\infty}{\lim}r_{m,N}=R_N ,$    where $R_N$ is the solution of the equation $2r^N=1-r, $
and $\underset{m,N\rightarrow\infty}{\lim}R_{m,N}= 1$.
\item[(ii)] Let $f=h+\overline{g}$ in $S_H^0$ be a stable univalent harmonic mapping on the unit disk $U$, then for each $m, N \in \aleph$
we have
 $$|f(z^m)| +  \underset{n=N}{\overset{\infty}{{\displaystyle\sum}}}\left( |a_{n}|+|b_n|\right)r^n \leq d(f(0),\partial f(U))$$ if $|z|\leq r_{m,N},$
where $r_{m,n}$ is the unique root in $(0,1)$ of  the function

\begin{equation}\label{functionstarlike1}
\psi_{m,N} (r)=\frac{r^m}{(1-r^m)^2}+\frac{r^{N}(1 -Nr+r)}{(1-r)^2}-\frac{1}{4},\\
\end{equation}
and $r_{m,N}$ cannot be improved. Here $\underset{N\rightarrow\infty}{\lim}r_{m,N}=\left(\frac{1}{5}\right)^{1/m} $,   $\underset{m\rightarrow\infty}{\lim}r_{m,N}=r_N ,$  where $r_N$ is the solution of the equation $$4r^N(N-Nr+1)=(1-r)^2, $$ and $\underset{m,N\rightarrow\infty}{\lim}r_{m,N}= 1$.

Moreover, $$|f(z^m)|^2 +  \underset{n=N}{\overset{\infty}{{\displaystyle\sum}}}\left( |a_{n}|+|b_n|\right)r^n \leq d(f(0),\partial f(U))$$ if $|z|\leq r_{m,N},$
where $R_{m,n}$ is the unique root in $(0,1)$ of  the function

\begin{equation}\label{fucntionstarlike2}
\kappa_{m,N} (r)=\frac{r^{2m}}{(1-r^m)^4}+\frac{r^{N}(1 -Nr+r)}{(1-r)^2}-\frac{1}{4},
\end{equation}
and $R_{m,N}$ cannot be improved.
\end{enumerate}
\end{theorem}

\begin{remark}
We note that for $N=1$, we have $r_1 =1/3 $ for the stable convex harmonic case  and $r_1 = 0.17 $ which is consistent with Bohr's inequality obtained in Theorem \ref{HarmonicBohr}.

\end{remark}
\begin{proof}
If $f$ is  stable univalent (or stable starlike), Theorem A gives for $|z|=r$,

\begin{equation*}
|f(z^m)|   \leq \frac{r^m}{1-r^m},
\end{equation*}

it follows that
\begin{equation*}
|f(z^m)| +   \underset{n=N}{\overset{\infty}{{\displaystyle\sum}}}\left( |a_{n}|+|b_n|\right)r^n \leq 
\frac{r^m}{1-r^m}+\sum\limits_{n=N}^{\infty }r^{n} = \frac{r^m}{1-r^m}+ \frac{r^N}{1-r}.
\end{equation*}

Hence, using equation \eqref{dist1},  $|f(z^m)| +   \underset{n=N}{\overset{\infty}{{\displaystyle\sum}}}\left( |a_{n}|+|b_n|\right)r^n \leq d(f(0 ) ,\partial f(U))$ for $|z| \leq R_{m,N},$ where $R_{m,N}$ is
the unique root in $(0,1)$ of the function   $ H_{m,N} (r) $ defined in (\ref{functionconvex1}).

Similarly, $|f(z^m)|^2 +   \underset{n=N}{\overset{\infty}{{\displaystyle\sum}}}\left( |a_{n}|+|b_n|\right)r^n\leq \frac{r^{2m}}{(1-r^m)^2}+ \frac{r^N}{1-r} \leq d(f(0),\partial f(U))$ if \ $|z|\leq r_{m,n}$ where $r_{m,n}$ is the unique root in $(0,1)$ of the function   $ K_{m,N} (r) $ defined in (\ref{functionconvex2}).

We note that  both functions $ H_{m,N} (r) $ and  $ K_{m,N} (r) $ are continuous in $r$ and  their value at zero is negative ( equals to $-1$)  and at $1$ is positive (equals to 2) which guarantees the existence of  roots.

The fact that this is the best possible radii is again guaranteed by a suitable rotation of the analytic function $l(z) = \frac{z}{1-z}$.

Part (ii) follows in a similar fashion using part (ii) of the Theorems A and B. 
Sharpness is obtained by using the analytic koebe function $k(z)=\frac{z}{(1-z)^2}.$ 
\end{proof}

As an immediate corollary letting $m=1$ we get:
\begin{corollary}
\begin{enumerate}
\item[(i)] Let $f=h+\overline{g}$ in $S_H^0$  be a stable convex harmonic mapping on the unit disk U, then for each $ N \in \mathbb{N},$ we have
 $$|f(z)| +  \underset{n=N}{\overset{\infty}{{\displaystyle\sum}}}\left( |a_{n}|+|b_n|\right)r^n \leq d(f(0),\partial f(U))$$ if $|z|\leq r_{N},$
where $r_{N}$ is the unique root in $(0,1)$ of  the the function
\begin{equation*}
  H_{N} (r) = 2r^N+3r-1,
\end{equation*}

and $r_{N}$ cannot be improved. 

Moreover, $$|f(z)|^2 +  \underset{n=N}{\overset{\infty}{{\displaystyle\sum}}}\left( |a_{n}|+|b_n|\right)r^n \leq d(f(0),\partial f(U))$$ if $|z|\leq R_{N},$
where $R_{N}$ is the unique root in $(0,1)$ of  the the function
\begin{equation*}
  K_{N} (r) = 2r^N(1-r)+r^{2}+ 2r-1,
\end{equation*}
and $R_{N} $ cannot be improved.
\item[(ii)] Let $f=h+\overline{g}$ in $S_H^0$ be a stable univalent harmonic mapping on the unit disk $U$, then for each $ N \in \mathbb{N}$
we have
 $$|f(z)| +  \underset{n=N}{\overset{\infty}{{\displaystyle\sum}}}\left( |a_{n}|+|b_n|\right)r^n \leq d(f(0),\partial f(U))$$ if $|z|\leq r_{N},$
where $r_{n}$ is the unique root in $(0,1)$ of  the function

$$\psi_{N} (r)=r+r^{N}(1 -Nr+r)-\frac{1}{4}(1-r)^2.$$

Moreover,  $$|f(z)|^2 +  \underset{n=N}{\overset{\infty}{{\displaystyle\sum}}}\left( |a_{n}|+|b_n|\right)r^n \leq d(f(0),\partial f(U))$$ if $|z|\leq r_{N},$
where $r_{N}$ is the unique root in $(0,1)$ of  the function

$$\kappa_{N} (r)=\frac{r^2}{(1-r)^4}+\frac{r^{N}(1 -Nr+r)}{(1-r)^2} - 1/4.$$

\end{enumerate}

\end{corollary}
\begin{remark}
We note that  in case (i) for $N=1$, we have $R_1 =1/5 $ for the stable convex harmonic case  and $r_1 = 2-\sqrt{3} $.

\end{remark}
\section{Bohr's phenomenon for stable logharmonic mappings}

In this section we get Bohr's radius in the case of stable logharmonic mappings. For $f(z)=zh\overline{g}$,  such that $f (0) = 0 $ and $h(0) = g(0) = 1,$ we let  
\begin{equation*}
h(z)=exp\left( \sum\limits_{n=1}^{\infty }a_{n}z^{n}\right) and\
g(z)=exp\left( \sum\limits_{n=1}^{\infty }b_{n}z^{n}\right) .
\end{equation*} 

\subsection{Bohr's inequality for stable logharmonic mappings }
We define the majorant series for  $f=zh\overline{g}$  as in \cite{AAC}  to be 
\begin{equation*}
B_{f}(r)=|z|exp\left( \sum\limits_{n=1}^{\infty
}|a_{n}+e^{it}b_{n}||z|^{n}\right) 
\end{equation*}%
In \cite{AAC}, the authors proved that \ for  $f(z)=zh\overline{g}\ $%
starlike logharmonic mappings $B_{f}(r)\leq d(0,\partial f(U))$ for $\
|z|\leq r_{0},$ where $r_{0}$ $\approx 0.09078.$ We show here that the
results improve under the stability assumption.

We will be needing the following Distortion Theorem and Coefficient Estimates proved in \cite{AE}.

\begin{thmd}(\cite{AE})\label{distortion2}
\qquad Let $f=zh\overline{g}$ be a stable univalent (or resp. stable starlike) logharmonic mapping on the unit disk U, 0$\notin hg(U).$%
Then for all  $z \in U,$ the following inequalities hold:%
\begin{equation*}
\dfrac{|z|}{(1+|z|)^{2}}\leq |f(z)|\leq \dfrac{|z|}{(1-|z|)^{2}}.
\end{equation*}
\end{thmd}

Moreover we have the following coefficient estimates :

\begin{thme}(\cite{AE})\label{coefficient2}
 Assume that  $f(z)=zh\overline{g}\ $ is stable univalent
logharmonic mapping. Then for all nonnegative integers $n$, we have 
\begin{equation*}
\left\vert |a_{n}|-|b_{n}|\right\vert \leq \max \{|a_{n}|,|b_{n}|\}\leq
|a_{n}|+|b_{n}|\leq n.
\end{equation*}
\end{thme}

We will next state and prove a version of Bohr's inequality for stable logharmonic mappings.
\begin{theorem}\label{logbohr}
Let $f=zh\overline{g}$ be a stable univalent (or resp. stable starlike)
logharmonic mapping on the unit disk U, 0$\notin hg(U)$. Then 

\begin{enumerate}
\item[(i)]  $B_{f}(r)<1$ if \ $|z|\leq r_{0}=0.378,$ where $r_{0}$ is the unique root in $%
(0,1)$ of

\begin{equation*}
r\exp \left( \frac{r}{\left( 1-r\right) ^{2}}\right) = 1.
\end{equation*}

\item[(ii)] $B_{f}(r)\leq d(0,\partial f(U))$ if \ $|z|\leq r_{0}=0.286,$
where $r_{0}$ is the unique root in $(0,1)$ of \ 
\begin{equation*}
r\exp \left( \frac{r}{\left( 1-r\right) ^{2}}\right) =\frac{1}{4}.
\end{equation*}
\end{enumerate}
\end{theorem}

\begin{proof}
For $|z|=r$, we have 
\begin{equation*}
B_{f}(r)\leq r\ exp\left( \sum\limits_{n=1}^{\infty }\left(|a_{n}|+|b_{n}|\right) r^{n}\right) \leq r\ exp\left(\sum\limits_{n=1}^{\infty }nr^{n}\right) =r\exp \left( \frac{r}{\left(1-r\right) ^{2}}\right).
\end{equation*}

It follows that $B_{f}(r)<1~$ for $|z|\ \leq r_{0}=0.378,$ where $r_{0}$ is
the unique root in $(0,1)$ of \ 

\begin{equation*}
r\exp \left( \frac{r}{\left( 1-r\right) ^{2}}\right) =1.
\end{equation*}

We next use Theorem C to establish bounds for $d(0,\partial f(U))$ in
the case where  $f $ is stable univalent or stable starlike.
\begin{equation}\label{distance2}
d(0, \partial f(U))=\underset{|z|\rightarrow 1}{\lim }\inf \ |f(z)-f(0)|=%
\underset{|z|\rightarrow 1}{\lim }\inf \ \dfrac{|f(z)-f(0)|}{|z|}\geq \underset{|z|\rightarrow 1}{\lim }\inf \dfrac{|z|}{(1+|z|)^{2}}= \frac{1}{4}.
\end{equation}%
Moreover, $B_{f}(r)<\frac{1}{4}\leq d(0,\partial f(U))$ if \ $|z|\leq
r_{0}=0.286,$ where $r_{0}$ is the unique root in $(0,1)$ of \ 
\begin{equation*}
r\exp \left( \frac{r}{\left( 1-r\right) ^{2}}\right) =\frac{1}{4}.
\end{equation*}
\end{proof}

\subsection{Improved Bohr inequality for logharmonic mappings}

We next show an improved version of Bohr inequality for stable logharmonic mappings that follows the ideas proposed in \cite{EPR,KP4}.

\begin{theorem}
Let $f=zh\overline{g}$ be a stable univalent (or resp. stable starlike)
logharmonic mapping on the unit disk U, 0$\notin hg(U)$. Then 
 $$|z|exp\left( \underset{n=1}{\overset{\infty}{{\displaystyle\sum}}}( |a_{n}+e^{it}b_n|+|a_n||b_n| )|z|^n\right) \leq d(0,\partial f(U)),$$ if \ $|z|\leq r_{0}=0.152,$
where $r_{0}$ is the unique root in $(0,1)$ of \ 
\begin{equation*}
r\exp \left( \frac{r}{\left( 1-r\right) ^{2}}-\frac{r(r+1)}{(r-1)^3}\right) =\frac{1}{4}.
\end{equation*}

\end{theorem}

\begin{proof}

A simple computation along Theorem D give: 
\begin{align*}
&|z|exp\left( \underset{n=1}{\overset{\infty}{{\displaystyle\sum}}}( |a_{n}+e^{it}b_n|+|a_n||b_n| )|z|^n\right)\\
&\leq r\ exp\left( \sum\limits_{n=1}^{\infty }\left(|a_{n}|+|b_{n}|+|a_n||b_n|\right) r^{n}\right) \\
&\leq r\ exp\left(\sum\limits_{n=1}^{\infty }nr^{n}+\sum\limits_{n=1}^{\infty }n^2r^{n}\right) \\
&=r\exp \left( \frac{r}{\left(1-r\right) ^{2}}-\frac{r(r+1)}{(r-1)^3}\right),\\
\end{align*}
which implies that $|z|exp\left( \underset{n=1}{\overset{\infty}{{\displaystyle\sum}}}( |a_{n}+e^{it}b_n|+|a_n||b_n| )|z|^n\right) \leq d(0,\partial f(U))$ if \ $|z|\leq r_{0}=0.152,$
where $r_{0}$ is the unique root in $(0,1)$ of \ 
\begin{equation*}
r\exp \left( \frac{r}{\left( 1-r\right) ^{2}}-\frac{r(r+1)}{(r-1)^3}\right) =\frac{1}{4}. 
\end{equation*}

Sharpness follows from suitable rotation of the koebe function $f(z) =\frac{z}{(1-z)^2}.$ (??)

\end{proof}
\subsection{Refined Bohr inequality for stable logharmonic mappings}

We next state a refined version of Bohr's inequality for stable logharmonic mappings. The proof is omitted as the idea is similar to the previous theorems in this section.
\begin{theorem}
Let $f=zh\overline{g}$ be a stable univalent (or resp. stable starlike)
logharmonic mapping on the unit disk U, 0$\notin hg(U)$. Then 

 $$|z|exp\left( \underset{n=1}{\overset{\infty}{{\displaystyle\sum}}}( |a_{n}+e^{it}b_n||z|^n+ |a_{n}+e^{it}b_n|^2|z|^{2n})\right) \leq d(0,\partial f(U))$$ if \ $|z|\leq r_{0}=0.271,$
where $r_{0}$ is the unique root in $(0,1)$ of \ 
\begin{equation*}
r\exp \left( \frac{r}{\left( 1-r\right) ^{2}}-\frac{r^4+r^2}{(r^2-1)^3}\right) =\frac{1}{4}.
\end{equation*}

\end{theorem}

\subsection{Bohr Rogosinski radius for stable logharmonic mappings}

\begin{theorem}
Let $f=zh\overline{g}$  be a stable univalent logharmonic mapping on the unit disk $U$, then for each $m, N \in \mathbb{N}$
we have
 $$|f(z^m)| + |z|exp\left( \underset{n=N}{\overset{\infty}{{\displaystyle\sum}}}( |a_{n}+e^{it}b_n|)r^n\right) \leq d(f(0),\partial f(U))$$ if $|z|\leq r_{m,N},$
where $r_{m,n}$ is the unique root in $(0,1)$ of  the function

\begin{equation}\label{functionlog}
\psi_{m,N} (r)=\frac{r^m}{(1-r^m)^2}+r exp \left(\frac{r^{N}(N -Nr+r)}{(1-r)^2}\right)-\frac{1}{4},\\
\end{equation}
and $r_{m,N}$ cannot be improved. Here $\underset{N\rightarrow\infty}{\lim}r_{m,N}=a_m ,$ where $a_m$ solves the equation $$ \frac{r^m}{(1-r^m)^2} + r=1/4$$ \,\, and   $\underset{m\rightarrow\infty}{\lim}r_{m,N}=r_N ,$  where $r_N$ is the solution of the equation $$rexp \left(\frac{r^{N}(N -Nr+r)}{(1-r)^2}\right)=1/4, $$ and $\underset{m,N\rightarrow\infty}{\lim}r_{m,N}= 1/4$.

Moreover. $$|f(z^m)|^2 +  |z|exp\left( \underset{n=N}{\overset{\infty}{{\displaystyle\sum}}}( |a_{n}+e^{it}b_n|)r^n\right)  \leq d(f(0),\partial f(U))$$ if $|z|\leq R_{m,N},$
where $R_{m,n}$ is the unique root in $(0,1)$ of  the function

\begin{equation}
\kappa_{m,N} (r)=\frac{r^{2m}}{(1-r^m)^4}+r exp \left(\frac{r^{N}(N -Nr+r)}{(1-r)^2}\right)-\frac{1}{4},
\end{equation}
and $R_{m,N}$ cannot be improved. Here $\underset{N\rightarrow\infty}{\lim}R_{m,N}=a_m ,$ where $a_m$ solves the equation $$ \frac{r^{2m}}{(1-r^m)^4} + r=1/4$$ \,\, and   $\underset{m\rightarrow\infty}{\lim}R_{m,N}=r_N ,$  where $r_N$ is the solution of the equation $$rexp \left(\frac{r^{N}(N -Nr+r)}{(1-r)^2}\right)=1/4, $$ and $\underset{m,N\rightarrow\infty}{\lim}r_{m,N}= 1/4$.

\end{theorem}
\begin{proof}
We use Theorem C and D to get that
\begin{align*}
&|f(z^m)| + |z|exp\left( \underset{n=N}{\overset{\infty}{{\displaystyle\sum}}}( |a_{n}+e^{it}b_n|)r^n\right)\\
& \leq  \frac{r^m}{(1-r^m)^2}+ r exp \left(\frac{r^{N}(N -Nr+r)}{(1-r)^2}\right)\\
& \leq d(f(0),\partial f(U))
\end{align*}
 if $|z|\leq r_{m,N},$
where $r_{m,n}$ is the unique root in $(0,1)$ of  the function defined in equation \eqref{functionlog}.

The second inequality follows in a similar fashion.
\end{proof}

\begin{remark}
We note that for N=1, we have for large $m$, $r_1 \approx 0.286 $ which is consistent with Theorem \ref{Logbohr}. 

\end{remark}
Taking $m=1$ we get the following corollary

\begin{corollary}
Let $f=zh\overline{g}$  be a stable univalent logharmonic mapping on the unit disk $U$, then for each $ N \in \aleph$
we have
 $$|f(z)| + |z|exp\left( \underset{n=N}{\overset{\infty}{{\displaystyle\sum}}}( |a_{n}+e^{it}b_n|)r^n\right) \leq d(f(0),\partial f(U))$$ if $|z|\leq r_{N},$
where $r_{N}$ is the unique root in $(0,1)$ of  the function

\begin{equation}\label{functionstarlike1}
\psi_{N} (r)=\frac{r}{(1-r)^2}+r exp \left(\frac{r^{N}(N -Nr+r)}{(1-r)^2}\right)-\frac{1}{4},\\
\end{equation}
and $r_{N}$ cannot be improved.

Moreover, $$|f(z)|^2 +  |z|exp\left( \underset{n=N}{\overset{\infty}{{\displaystyle\sum}}}( |a_{n}+e^{it}b_n|)r^n\right)  \leq d(f(0),\partial f(U))$$ if $|z|\leq R_{N},$
where $R_{N}$ is the unique root in $(0,1)$ of  the function

\begin{equation}\label{functionstarlike1}
\kappa_{N} (r)=\frac{r^2}{(1-r)^4}+r exp \left(\frac{r^{N}(N -Nr+r)}{(1-r)^2}\right)-\frac{1}{4},\\
\end{equation}
and $R_{N}$ cannot be improved.
\end{corollary}


\begin{thebibliography}{}
\bibitem{AA1} Z. Abdulhadi and Y.Abumuhanna, Starlike logharmonic mappings of
order alpha, Journal of Inequalities in Pure and Applied Mathematics. 7(4)
Art.123, (2006), 1-6.

\bibitem{AA2} Z. Abdulhadi\ and\ R. M. Ali, Univalent logharmonic mappings in
the plane, Abstr. Appl. Anal. 2012, Art. ID 721943, pp.1-32.

\bibitem{AB2} Z. Abdulhadi\ and\ D. Bshouty, Univalent functions in $H\cdot 
\overline{H}(D)$, Trans. Amer. Math. Soc. 305 (1988), no.~2, 841--849.

\bibitem{AH3} Z. Abdulhadi and W. Hengartner, One pointed univalent logharmonic
mappings, J. Math. Anal. Appl. 203 (1996), no.2, 333-351.

\bibitem{AH4} Z. Abdulhadi\ and\ W. Hengartner, Polynomials in $H\overline{H}$,
Complex Variables Theory Appl. 46 (2001), no.~2, 89--107.

\bibitem{AH2} Z. Abdulhadi, L. El Hajj, Stable geometrical properties of
logharmonic mappings, journal of complex variables and elliptic
equations,2018, volume 63,issue 6, 854-870,
\bibitem{A} Abu-Muhanna, Y., Bohr phenomenon in subordination and bounded harmonic classes, Complex Var. Elliptic Equ. 55 (11) (2010), 1071–1078.
\bibitem{AA} Abu-Muhanna, Y. and Ali, R. M., Bohr phenomenon for analytic functions into the exterior of a compact convex body, J. Math. Anal. Appl. 379 (2)(2010), 512–517.
\bibitem{AAH} Abu-Muhanna, Y., Ali R. M., Ng Z. C. and Hasni S. F. M. , Bohr radius for subordinating families of analytic functions and bounded harmonic mappings, J. Math. Anal. Appl. 420 (1)(2014), 124–136.
\bibitem{AAC} Ali R. , Abdulhadi Z. and   Chuan Ng  Z. , The Bohr radius for starlike logharmonic mappings, Complex Variables and Elliptic Equations, 61:1 (2016), 1-14.
\bibitem{A1} Aizenberg, L. , Multidimensional analogues of Bohr’s theorem on power series. Proc. Amer. Math. Soc. 128 (4)(2000), 1147–1155.
\bibitem{A2} Aizenberg, L. , Generalization of results about the Bohr radius for power series, Stud. Math. 180 (2007), 161–168.
\bibitem{A3} Aizenberg, L., Aytuna, A. and Djakov, P. , Generalization of a theorem of Bohr for bases in spaces of holomorphic functions of several complex variables, J. Math. Anal. Appl. 258 (2)(2001), 429–447
\bibitem{AAP}  R.M. Ali, Y. Abu-Muhanna, S. Ponnusamy, On the Bohr inequality,  N. K. Govil et al. (eds.) Progress in Approximation Theory and Applicable Complex Analysis, Springer Optimization and Its Applications 117 (2016), 265-295.
\bibitem{AB} R. M. Ali, R. W. Barnard and A. Yu. Solynin, A note on the Bohr's phenomenon for power series, J. Math. Anal. Appl. 449 (1) (2017), 154-167.
\bibitem{AKP} S. A. Alkhaleefah, I. R. Kayumov and S. Ponnusamy, On the Bohr inequality with a fixed zero coefficient, Proc. Amer. Math. Soc. 147 (12) (2019), 5263 - 5274.
\bibitem{AS} Y. Abu-muhanna, L. Shakaa, Bohr's inequality for large functions, (2020), arXiv:2010.07090.
\bibitem{AE} Z. AbdulHadi, L. El Hajj  Stable geometrical properties of logharmonic mappings, Complex Variables and Elliptic Equations, 63:6 (2018), 854-870.
\bibitem{B} M. B. Balk, Polyanalytic functions and their generalizations, Complex Analysis I Encyclopedia Math. Sci (85) Springer Berlin (1997), 195-253.
\bibitem{BDK}C. B\'en\'eteau, A. Dahlner and D. Khavinson, Remarks on the Bohr phenomenon, Comput. Methods Funct. Theory 4(1) (2004), 1–19.
\bibitem{BL} X-X. Bai, M.S. Liu, Landau-type theorems of poly-harmonic mappings and log-p-harmonic mappings,
 Complex Anal. Oper. Theory 13 (2) (2019), 321-340.
\bibitem{Bo} H. Bohr, A theorem concerning power series, Proc. London Math. Soc. 13 (2) (1914), 1-5.
\bibitem{Bo2} E. Bombieri, Sopra un teorema di H. Bohr e G. Ricci sulle funzioni maggioranti delle serie di potenze, Boll. Unione Mat. Ital. 17 (1962), 276-282.
\bibitem{DR} P. B. Djakov and M. S. Ramanujan, A remark on Bohr's theorems and its generalizations, J. Analysis 8 (2000), 65-77.
\bibitem{EPR} Evdoridis, S., Ponnusamy, S. and Rasila, A. (2017), A. Improved Bohr’s inequality for locally univalent harmonic mappings, Indag. Math. (N.S.) 30 (1) (2019), 201—213.

\bibitem{HM} R. Hernandez and M. J. Mart\i n, Stable geometric properties of analytic and harmonic functions, Proc. Cambridge Phil. Soc.,155(2013), 343--359
\bibitem{KKP2} A. Kayumova, I. R. Kayumov and S. Ponnusamy, Bohr’s inequality for harmonic mappings and beyond, Mathematics and computing, 245–256, Commun. Comput. Inf. Sci., 834, Springer, Singapore, 2018.
\bibitem{KKP} Ilgiz R. Kayumov, Diana M. Khammatova, Saminathan Ponnusamy, Bohr–Rogosinski phenomenon for analytic functions and Cesáro operators, journal of Mathematical Analysis and Applications, Volume 496, Issue 2 (2021), 124824.
\bibitem{KP1} I. R. Kayumov and S. Ponnusamy, Bohr inequality for odd analytic functions, Comput. Methods Funct. Theory 17 (4) (2017), 679-688.
\bibitem{KP5}I.R. Kayumov and S. Ponnusamy, Bohr-Rogosinski radius for analytic functions, preprint, see https://arxiv.org/abs/1708.05585
\bibitem{KP2} I. R. Kayumov and S. Ponnusamy, Bohr's inequality for analytic functions with lacunary series and
harmonic functions, J. Math. Anal. and Appl. 465 (2) (2018), 857-871.
\bibitem{KP3} I. R. Kayumov and S. Ponnusamy, On a powered Bohr inequality, Ann. Acad. Sci. Fenn. Ser. A I Math. 44 (2019), 301-310.
\bibitem {KP4} Kayumov, I. R. and Ponnusamy, S. (2018), Improved Version of Bohr’s Inequality, ComptesRendus Mathematique, .
\bibitem{KPS} I. R. Kayumov, S. Ponnusamy and N. Shakirov, Bohr radius for locally univalent harmonic mappings, Math. Nachr. 291 (2018), no. 11-12, 1757-1768.
\bibitem{IKKP} A. Ismagilov, A. Kayumova, I. R. Kayumov,  and S.Ponnusamy,   Bohr type inequalities in certain classes of analytic functions, Math. Sci (New York) (English)(2019)
\bibitem{LG}  E. Landau and D. Gaier, Darstellung und Begruundung einiger neuerer Ergebnisse der Funktion en theorie, Springer-Verlag, 1986.
\bibitem{LP} G. Liu and S. Ponnusamy, On harmonic $\nu$-Bloch and $\nu$-Bloch-type mappings, Results Math. 73(3) (2018), 73-90.
\bibitem{LP2}M. S. Liu and S. Ponnusamy, Multidimensional analogues of refined Bohr’s inequality, Proc.Amer. Math. Soc. 149(3) (2021), 2133–2146.
\bibitem{PPS} V.I. Paulsen, G. Popascu, D. Singh,  On Bohr's inequality, Proc. Lond. Math. Soc. 3(85) (2002), 493-512. 
\bibitem{PVW} S. Ponnusamy , R. Vijayakumar and  KJ. Wirths,  New Inequalities for the Coefficients of Unimodular Bounded Functions. Results Math 75, 107 (2020).
\bibitem{SS}  I. Schur und G. Szego, Uber die Abschnitte einer im Einheitskreise beschankten Potenzreihe, Sitz.
Ber. Preuss. Acad. Wiss. Berlin Phys.-Math. Kl. (1925), 545–560
\bibitem{S2} S. Sidon,  Uber einen satz von Herrn Bohr. Math. Z. 26 (1927), 731-732.
\bibitem{T} M. Tomic, Sur un th\'eor\`eme de H. Bohr. Math. Scand. 11 (1962) , 103-106.
\end{thebibliography}
\end{document}